\documentclass[reqno]{amsart}

\usepackage{amsmath}
\usepackage{amssymb}
\usepackage{amsfonts}
\usepackage{graphicx}
\usepackage{amsthm}
\usepackage{enumerate}
\usepackage{lscape}
\usepackage{dsfont}
\usepackage{color}
\usepackage{mathtools}

\usepackage{setspace}
\newcommand{\R}{\mathds{R}}

\newcommand{\Ric}{\mathrm{Ric}}

\newcommand{\de}{\partial}          

\newcommand{\K}{K\"{a}hler}

\newcommand{\tr}{\operatorname{tr}}
\newcommand{\ov}[1]{\overline{#1}}

\newcommand{\deb}{\ov\partial}

\newcommand{\lmb}{\lambda}

\newcommand{\Hol}{{\operatorname{Hol}}}

\newcommand{\B}{\operatorname{Berg}}
\newcommand{\Sz}{Szeg\"{o}}

\def\diag{\mathrm{diag\;}}

\def\a{\alpha}
\def\b{\beta}

\def\bZ{{\mathbb Z}}
\def\bC{{\mathbb C}}

\def\b1{{\rm id}}

\newtheorem{theor}{Theorem}[section]
\newtheorem{prop}[theor]{Proposition}
\newtheorem{defin}[theor]{Definition}
\newtheorem{lem}[theor]{Lemma}

\newtheorem{ex}[theor]{Example}
\newtheorem{remark}[theor]{Remark}

\begin{document}

\title[The log-term  of  the disc bundle over a homogeneous Hodge manifold]{The log-term of the disc bundle over a homogeneous Hodge manifold}

\author{Andrea Loi}
\address{(Andrea Loi) Dipartimento di Matematica \\
         Universit\`a di Cagliari (Italy)}
         \email{loi@unica.it}

\author{Roberto Mossa}
\address{(Roberto Mossa) Dipartimento di Matematica \\
         Universit\`a di Cagliari (Italy)}
         \email{roberto.mossa@gmail.com}

\author{Fabio Zuddas}
\address{(Fabio Zuddas) Dipartimento di Matematica e Informatica \\
          Via delle Scienze 206 \\
         Udine (Italy)}
\email{fabio.zuddas@uniud.it}

\thanks{
The first two  authors were  supported  by Prin 2010/11 -- Variet\`a reali e complesse: geometria, topologia e analisi armonica -- Italy and also by INdAM-GNSAGA - Gruppo Nazionale per le Strutture Algebriche, Geometriche e le loro Applicazioni; the third author was supported by the FIRB project 2012 ``Geometria Differenziale e Teoria geometrica delle funzioni".   }

\begin{abstract}
We show the vanishing of  the log-term in the Fefferman expansion of the Bergman kernel of 
the  disk bundle over  a  compact simply-connected  homogeneous \K\--Einstein  manifold of classical type.
Our results extends that  in \cite{englisramadanov} for the case of  Hermitian symmetric spaces of compact type.
\end{abstract}
 
\maketitle

\tableofcontents

\section{Introduction}
We recall the basic framework.
Let $\left( L,\, h\right) $ be a positive Hermitian line bundle over a compact  \K\ manifold $\left( M,\, g\right) $ of complex dimension $n$, such that $\Ric \left( h\right) =\omega_g$, where $\omega_g$ denotes the \K\ form associated to $g$
and $\Ric \left( h\right) $  is the two--form on $M$ whose
local expression is given by
\begin{equation}\label{prodherm}
\Ric \left( h\right) =-\frac{i}{2}
\partial\bar\partial\log h\left( \sigma (x) ,\, \sigma (x) \right),
\end{equation}
for a trivializing holomorphic section $\sigma :U\rightarrow
L\setminus\{0\}$. In the quantum mechanics terminology the pair $(L, h)$ is called a {\em geometric quantization} of $\left( M,\, \omega_g\right) $ and $L$ the {\em quantum line bundle}.
Notice that such an $h$ exists if and only if $\frac{\omega_g}{\pi}$ is an integral \K\ form which represents the first Chern class
$c_1(L)$ of $L$. 
Consider the negative Hermitian line bundle $\left( L^*,\, h^* \right) $ over $\left( M,\, g\right) $ dual   to $\left( L,\, h\right) $ and
let $D\subset L^*$ be the unit disc bundle over $M$, i.e.
\begin{equation}\label{diskbundle}
D=\{v\in L^* \ |\  \rho \left( v\right) :=1-h^*\left( v,\, v\right) >0\}.
\end{equation}
It is not hard to see (and well-known)  that the condition $\Ric \left( h\right) =\omega_g$ implies that $D$ is a strongly pseudoconvex   domain in $L^*$ with smooth boundary
$X=\partial D=\{v\in L^*\ |\ \rho \left( v\right) =0\}.$
Consider the Bergman  space $\mathcal B _{D}$ consisting  of holomorphic $\left( n +1, 0 \right)$-forms $\eta$ on $D$ such that 
$\frac{i^{n+1}}{2^{n+1}}\int_D \eta \wedge \ov \eta < \infty$
and the corresponding Bergman kernel  ${\mathcal \B_D}$ namely the $(n+1, n+1)$-form
given by:
\[
{\mathcal \B_D}=\sum_{j}\eta_j\wedge\ov\eta_j,
\]
where $\{\eta_j\}$ is an orthonormal basis of $\mathcal B _{D}$.

We say that the log-term of the Bergman kernel ${\mathcal \B_D}$
vanishes if there exists a non-vanishing  $(n+1, n+1)$-form  $a$ on $\ov D$
(the closure of $D$)
such that 
\begin{equation}\label{logtermzero}
{\mathcal \B_D}(v)=a(v)\rho(v)^{-n-2}, \ v\in D.
\end{equation}

The main goal of this paper is to study the  Bergman kernel ${\mathcal \B_D}$
of the disk bundle $D\subset L^*$, when $L$ is the anticanonical bundle $K^*$  over a  compact homogeneous \K\--Einstein manifold
$(M, g)$ (and hence $L^*=K$ is the canonical bundle).

The following theorem represents our main result.

\begin{theor}\label{mainteor}
Let $(M, g)$ be a compact and  simply-connected homogeneous \K\--Einstein manifold  of classical type with 
$\frac{\omega_g}{\pi} \in c_1(L)$, 
$L=K^*$,  and  $D\subset L^*$
be the corresponding disk bundle.
Then the log-term of the Bergman kernel ${\mathcal \B_D}$ vanishes.
\end{theor}

Problems and results of this kind go back to the celebrated Fefferman's theorem (\cite{fefferman}) about the expansion of the Bergman kernel for domains in $\bC^n$. Let us recall that, for a strongly pseudoconvex domain $D \subseteq \bC^n$ given by a defining function $\psi$, the Bergman kernel $B(z) = \sum |f_j(z)|^2$ of $D$ (where ${f_j}$ is an orthonormal basis for the Bergman space of holomorphic square integrable functions $f: D \rightarrow \bC$) admits the following decomposition
\begin{equation}\label{fefferman}
B(z) = \phi(z) \psi(z)^{-n-1} + \tilde \phi(z) \log \psi(z),
\end{equation}
where $\phi$ does not vanish on $\partial D$ and $\tilde \phi$ is said to be the {\em log-term}. 
 In  \cite{ramadanov}, Ramadanov conjectures that if the log-term of $D$ vanishes, then $D$ is biholomorphic to the unit ball in $\bC^n$. The conjecture has been proved in some special cases, among which domains in $\bC^2$ and domains with rotational symmetries (see, for example, \cite{hira}, \cite{naka}). A decomposition analogous to (\ref{fefferman}) has been proved by Boutet de Monvel and Sj\"{o}strand for the Szeg\"{o} kernel (\cite{BoutSjos}), and there is a corresponding version of the Ramadanov conjecture.
 It is then natural to consider the analogous definitions and questions for strictly pseudoconvex
domains in complex manifolds. Our Theorem \ref{mainteor} extends  that  of M. Engli\v{s} and G.  Zhang \cite{englisramadanov} when $M$ is an Hermitian symmetric space of compact type.  In  that paper they ask (see Question 4 at page 911) if the vanishing of the log-term of the Bergman kernel of the disk bundle
over a compact \K\ manifold $(M, g)$
should imply that  the manifold is  symmetric.
Theorem \ref{mainteor} provides a negative answer to this question:
one can  simply take  a compact and simply-connected homogeneous \K\--Einstein manifold  of classical type
which is not symmetric.\footnote{
It  is worth mentioning that in   \cite{ALZ}  (see also \cite{englisramadanov})  it is proven the analogous of Theorem \ref{mainteor}
for the \Sz\  kernel, namely  the  vanishing of  the  log-term of the  \Sz\  kernel of the disk bundle over  a  compact homogenous Hodge manifold (not necessarily of classical type).}

The proof of Theorem \ref{mainteor} is based on  the following two facts satisfied by any homogeneous \K\--Einstein manifold as in the theorem:
\begin{itemize}
\item
the quantization bundle $L=K^*\rightarrow M$ is regular, namely its associated Kempf distortion function 
$T_{mg}$ is a  positive constant for all positive integer $m$ (see Lemma \ref{homreg} below);
\item
$(M, g)$ admits a Calabi's diastasis function
whose maximal domain of definition is an open and dense subset of $M$ biholomorphically equivalent to $\bC^n$,
where $n$ is the complex dimension of $M$ (see Theorem \ref{diastasis}).
\end{itemize}

We think that the result  of Theorem \ref{mainteor}
can be extended also for the disk bundle $D_m\subset L^{*m}$ associated to the   \K\--Einstein
form  $\frac{m\omega_g}{\pi}\in c_1(L^m)$, for any $m\geq 1$. 
Actually  we believe that the result holds true for all homogeneous \K\ metrics  not necessarily Einstein  or of classical type.

The paper is organized as follows.
In Section \ref{flag}  we are going to give a self-contained exposition of the results  proved in \cite{A-P2} about  $G$-invariant K\"ahler metrics on homogeneous Hodge manifolds $G/K$ of classical type  and on the existence of an open dense subset  $F_{reg}$ of $G/K$ biholomorphically equivalent to the Euclidean space. 
In  Section \ref{complcoord}, after recalling the construction of an explicit \K\ potential in $F_{reg}$ due to \cite{A-P2},  we will prove 
Theorem \ref{diastasis}, namely that this potential coincides with Calabi's diastasis function having $F_{reg}$ as  its maximal domain of definition.
Finally, Section  \ref{secmain} is dedicated to the proof of Theorem \ref{mainteor}.
The paper ends with an  Appendix which contains  the basic material on semisimple Lie algebras used in this paper.

\section{K\"ahler structures and complex coordinates of flag manifolds} \label{flag}
 The material  of this section is based on   \cite{A-P} and \cite{A-P2}. 
Here and below, $G$ will denote a compact semisimple group with Lie algebra $\mathfrak{g}$ and $G^{\bC}$, $\mathfrak{g}^{\bC}$ the corresponding complexifications.
 Given $Z \in \mathfrak{g}$, let us consider the orbit $F =\mathrm{Ad}_G Z$ of $Z$ for the adjoint action of $G$ on $\mathfrak{g}$. Then $F$ is diffeomorphic to the quotient manifold $G/K$, being $K$ the stabilizer of $Z$ with respect to the adjoint action, and is called a {\it flag manifold}. 

Recall (see, for example, \cite{bes}) that
each compact homogeneous \K\ manifold $M$ is the \K\ product of a flat complex torus and a simply-connected compact homogeneous \K\ manifold, and admits a \K-Einstein structure if and only if is a torus or is simply-connected.
In the simply-connected case, $M$ is isomorphic, as a homogeneous complex manifold, to an orbit of the adjoint action of its connected group of isometries $G$ (which, being compact and with no center, is semisimple).
In this  paper we restrict to the simply-connected case.


We are going to recall how one can describe combinatorially all the $G$-invariant complex structures on a flag manifold $F$ via root systems and Dynkin diagrams. In what follows, $\mathfrak{k}$ and $\mathfrak{k}^{\bC}$ denote respectively the Lie algebra of $K$ and its complexification. 
As it is known from the theory of complex semisimple Lie algebras, given a Cartan subalgebra $\mathfrak{h}^{\bC}$ of $\mathfrak{g}^{\bC}$ (i.e. a maximal abelian subalgebra such that $ad_X$ is diagonalizable for each $X \in \mathfrak{h}^{\bC}$), one has the decomposition
\begin{equation}\label{decomposition}
\mathfrak{g}^{\bC} = \mathfrak{h}^{\bC} + \sum_{\alpha \in R} \bC E_{\alpha},
\end{equation}
 where a {\it root} $\alpha \in R$ is a functional $\mathfrak{h}^{\bC} \rightarrow \bC$ such that $[H, E_{\alpha}] = \alpha(H) E_{\alpha}$ for each $H \in \mathfrak{h}^{\bC}$ and for some $E_{\alpha} \in \mathfrak{g}^{\bC}$ (called {\it root vector} of $\alpha$). The set $R$ of roots is called the {\it root system} of $\mathfrak{g}^{\bC}$. Recall also that any two Cartan subalgebras $\mathfrak{h}^{\bC}_1$ and $\mathfrak{h}^{\bC}_2$ are conjugate, i.e. there exists $g \in G^{\bC}$ such that $Ad_g(\mathfrak{h}^{\bC}_1) = \mathfrak{h}^{\bC}_2$.
 Since by definition of semisimple algebra the Killing form $B$ of $\mathfrak{g}^{\bC}$ is nondegenerate, to every root $\alpha \in R$ is associated by duality a vector $H_{\alpha} \in \mathfrak{h}^{\bC}$ which satisfies $B(H, H_{\alpha}) = \alpha(H)$ for every $H \in \mathfrak{h}^{\bC}$. The real vector space $\mathfrak{h}$ spanned by the $H_{\alpha}$'s, $\alpha \in R$, is a real form of $\mathfrak{h}^{\bC}$ on which $B$ is real and positive definite (Theorem 4.4 in \cite{helgason}). We can then define a scalar product between the roots by $\langle \alpha, \beta \rangle := B(H_{\alpha}, H_{\beta})$.

 A {\it basis} $\Pi$ of the root system is a subset $\Pi \subseteq R$ such that every root $\alpha \in R$ can be written as a linear combination of the elements of $\Pi$ with the coefficients either all non-negative or all non-positive. In the first (resp. second) case, $\alpha$ is said to be positive (resp. negative). The set of positive roots with respect to a fixed basis will be denoted by $R^+$. It can be shown that any subset $R^+ \subseteq R$ such that $R = R^+ \cup (-R^+)$, $R^+ \cap (-R^+) = \emptyset$ and $\alpha + \beta \in R^+$ provided $\alpha, \beta \in R^+$, $\alpha + \beta \in R$, is the set of the positive roots with respect to some basis.

The Lie algebra $\mathfrak{g}^{\bC}$ can be combinatorially represented by a {\it Dynkin diagram}, which is constructed as follows: fixed a basis $\Pi$ of $R$, the Dynkin diagram is the linear graph having $\sharp \Pi$ vertices, one for each element of $\Pi$, and such that the number of edges between two vertices depends on the value of the scalar product between the corresponding roots (more precisely, if the angle between the roots is $\theta$, the number of edges is equal to $4 \cos^2 \theta$, and one proves that this number can be only 0,1,2,3).
Moreover, an edge between two vertices is oriented with an arrow if and only if the corresponding roots have different norms (and the arrow goes from the longer root to the shorter one).
In fact, the diagram does not depend on the choice of the basis. The choice of a basis is called an {\it equipment} for the diagram. In the Appendix, we give for each of the classical groups $G = SU(d), Sp(d), SO(2d), SO(2d+1)$ an equipment which will have a fundamental role in the proofs of our results (see Remark \ref{remarkalternativeQ}).

Now, the complexification $\mathfrak{k}^{\bC}$ of the Lie algebra of the stabilizer subgroup $K$ turns out to be reductive (i.e. it decomposes as the direct sum of its center and a semisimple part) so we have

\begin{equation}\label{decompositionK}
\mathfrak{k}^{\bC} = Z(\mathfrak{k}^{\bC}) \oplus (\mathfrak{h}^{\bC})' + \sum_{\alpha \in R_K} \bC E_{\alpha}
\end{equation}
where $Z(\mathfrak{k}^{\bC})$ is the center of $\mathfrak{k}^{\bC}$ and $(\mathfrak{h}^{\bC})'$, $R_K \subseteq R$ are respectively a Cartan subalgebra and the root system of the semisimple part of $\mathfrak{k}^{\bC}$ (with respect to $(\mathfrak{h}^{\bC})'$).

The elements of $R_K$ (resp. of the complementary subset $R_M$) are usually called {\it white roots} (resp. {\it black roots}). The reason is that one can represent the flag manifold $G/K$ on the Dynkin diagram of $G$, equipped with a given basis $\Pi$, by painting black the vertices corresponding to roots belonging to $R_M$. One gets a decomposition $\Pi = \Pi_K \cup \Pi_M$ of the basis $\Pi$ and the resulting diagram is called {\it painted Dynkin diagram}.\footnote{Some authors (for example \cite{BFR}) reverse the notation and paint black the roots in $R_K$.}. Looking at a painted diagram, one can easily recover the root decomposition and the flag manifold: indeed, a root $\alpha$ belongs to $R_M$ if and only if $\alpha = \sum_{\beta \in \Pi} c_{\beta} \beta$ with $c_{\beta} \neq 0$ for some $\beta \in \Pi_M$; moreover, deleting the black nodes from the diagram one gets the Dynkin diagram of the semisimple part of $K$.

\begin{defin}\label{defQ}
A subset $Q \subseteq R_M$ is said to be maximal closed nonsymmetric if it satisfies the following conditions:

\begin{enumerate}

\item[(i)] $Q \cup (-Q) = R_M$;

\item[(ii)] $Q \cap (-Q) = \emptyset$;

\item[(iii)] for any $\alpha, \beta \in Q$ such that $\alpha + \beta \in R$ one has $\alpha + \beta \in Q$.

\end{enumerate}

\end{defin}

Then we have

\begin{prop}\label{Qcomplex}(Corollary 3.1 in \cite{A-P2})
There exists a one-to-one correspondence between $G$-invariant complex structures on $G/K$ and maximal closed nonsymmetric subsets $Q$ in $R_M$. 
 More precisely, given the decomposition $\mathfrak{g}^{\bC} = \mathfrak{k}^{\bC} + \mathfrak{m}^{\bC}$, $\mathfrak{m}^{\bC} = \sum_{\alpha \in R_M} \bC E_{\alpha}$, the $G$-invariant complex structure $J_Q$ associated to $Q$ is determined (on the complexified tangent space at the base point $o = K$, naturally identified with $\mathfrak{m}^{\bC}$) by $(J_Q)_o(E_{\pm \alpha}) = \pm i E_{\alpha}$ ($\alpha \in Q$).
\end{prop}

The following remark is crucial for the results of the next section.

\begin{remark}\label{remarkalternativeQ}
\rm Let $F = G/K$ be a flag manifold endowed with an invariant complex structure. We claim that, fixed any basis $\Pi$ for the root system $R$ of $G$, one can find a painting of the Dynkin diagram of $G$ equipped with $\Pi$ so that the associated flag manifold, endowed with the complex structure determined by the maximal closed nonsymmetric subset $R_M^+$ of the black roots which are positive with respect to $\Pi$ is $G$-diffeomorphic to $F$.

 In order to see that, fix a basis $\Pi_K$ for the set $R_K$ of white roots of $F$, and let $R_K^+$ be the induced set of white positive roots. If $Q$ is the maximal closed nonsymmetric subset of $R_M$ corresponding to the complex structure of $F$, take $R^+ := Q \cup R_K^+$ as the set of positive roots in $R$ (in fact, it is easily seen to satisfy the three properties needed to be the set of positive roots induced by a basis) and let $\Pi'$ be the basis of $R$ which induces $R^+$. Clearly, $\Pi'$ contains $\Pi_K$. Now, the Dynkin diagram of $G$ equipped with $\Pi'$ and with nodes in $\Pi' \setminus \Pi_K$ painted black represents $F$ and $Q$ is the set of positive black roots with respect to this equipment.
 Then, it is enough to apply to this painted diagram a Weyl group transformation $w$ (see for example Section II.6 in \cite{knapp}) such that $w(\Pi') = \Pi$ in order to get a repainted diagram equipped with $\Pi$ and such that $Q$ is sent to the set of black roots which are positive with respect to $\Pi$, as claimed.
By this remark, in the results we are going to prove about the flag manifolds $G/K$ of the classical groups $G = SU(d), SO(2d), SO(2d+1), Sp(d)$, it is not restrictive to assume that $G/K$ is represented by a painted Dynkin diagram equipped with the canonical equipment $\Pi_{can}$ (see the Appendix)  and with the invariant complex structure of $G/K$ given by the set $R_M^+$ of black positive roots (with respect to $\Pi_{can}$).

\end{remark}

Let now $\Pi$ be a basis of $R$ and let $\Pi = \Pi_K \cup \Pi_M$ be the decomposition into white and black nodes, with $\Pi_K = \{ \beta_1, \dots, \beta_k\}$, $\Pi_M = \{ \a_1, \dots, \a_m\}$. As we have recalled at the beginning of this section, the Killing form induces by duality a scalar product $\langle, \rangle$ on the real space ${\mathfrak{h}}^*$ spanned by the roots. The {\it fundamental weight} $\bar \a_i$
associated with $\a_i$, $i = 1, \dots, m$ is the element of ${\mathfrak{h}}^*$ defined by

\begin{equation}\label{deffundwe}
\frac{2 \langle \bar \alpha_i, \alpha_j \rangle}{\| \alpha_j \|^2} = \delta_{ij}, \ \ \langle \bar \alpha_i, \beta_j \rangle = 0
\end{equation}

If we denote by $\mathfrak{t} = Z(\mathfrak{k}) \cap \mathfrak{h}$ the intersection between the center $Z(\mathfrak{k})$ of $\mathfrak{k}$ and $\mathfrak{h}$, then the fundamental weights form a basis of the real space $\mathfrak{t}^*$.

\begin{prop}\label{transgression}(\cite{A-P2}, Proposition 2.2) There exists a natural isomorphism  between $\mathfrak{t}^*$ and the space of closed invariant 2-forms on $F$ given by

\begin{equation}\label{omegaxi}
\xi \in \mathfrak{t}^* \rightarrow \omega_{\xi} = \frac{i}{2 \pi} \sum_{\alpha \in R_M} \frac{2 \langle \xi, \alpha \rangle}{\langle \alpha, \alpha \rangle} \omega^{\alpha} \wedge \omega^{-\alpha}
\end{equation}
 where the $\omega^{\alpha}$'s are 1-forms on  $\mathfrak{m}^{\bC}$, dual to the $E_{\alpha}$'s.

If $J_Q$ is the complex structure associated to $Q = R_M^+$ with the given equipment $\Pi$, then $\omega_{\xi}$ is K\"ahler with respect to $J_Q$ if and only if all the coordinates of $\xi$ with respect to $\{ \bar \alpha_1, \dots, \bar \alpha_m \}$ are positive. Moreover, $\omega_{\xi}$ is integral if and only if the coordinates of $\xi$ with respect to $\{ \bar \alpha_1, \dots,  \bar \alpha_m\}$ are integers.
\end{prop}
 By Proposition \ref{Qcomplex}, invariant complex structures on $G/K$ are in one-to-one correspondence with maximal closed nonsymmetric subsets $Q \subseteq R_M$. Indeed, the manifold $G/K$ endowed with the complex structure $J_Q$ is biholomorphic to the complex homogeneous manifold $G^{\bC}/K^{\bC} G^{Q}$, where $G^{Q} = \exp(\mathfrak{g}^{Q})$ and $\mathfrak{g}^{Q} = \sum_{\alpha \in Q} \bC E_{\alpha}$ (see \cite{A-P}, \cite{A-P2} and also \cite{arv} for more details). Using this, in \cite{A-P2} it is given an explicit system of complex coordinates on an open subset of $G/K$ as follows.

Since the product $G_{reg}^{\bC} = G^{-Q} K^{\bC} G^{Q}$ (where $G^{-Q} = \exp(\mathfrak{g}^{-Q})$ and $\mathfrak{g}^{-Q} = \sum_{\alpha \in -Q} \bC E_{\alpha}$) defines an open dense subset in $G^{\bC}$, its image in $G^{\bC}/K^{\bC} G^{Q}$ via the natural projection $G^{\bC} \rightarrow G^{\bC}/K^{\bC} G^{Q}$ defines an open dense subset in $G/K$, denoted $F_{reg} = G_{reg}^{\bC}/K^{\bC} G^Q$. Clearly, $F_{reg} \simeq G^{-Q}$.
 Then, by

\begin{equation}\label{complexcoordinates}
z = (z_{\alpha})_{\alpha \in -Q} \in \bC^n \mapsto \exp(\sum_{\alpha \in -Q} z_{\alpha} E_{\alpha}) \in G^{-Q} \simeq F_{reg} \subseteq F
\end{equation}

(where $n$ is the cardinality of $Q$) one defines a system of complex coordinates on $F_{reg}$. 

\section{Calabi's diastasis function of classical flag manifolds} \label{complcoord}

We now  assume that $G$  is one of the classical groups $SU(d), Sp(d), SO(2d)$, $SO(2d+1)$ (where the orthogonal groups are realized as groups of matrices as in the Appendix).  In this section, after recalling the construction of an explicit    K\"ahler potential for the invariant K\"ahler form $\omega_{\xi}$ (in the notation of Proposition \ref{transgression}) in the  coordinates defined at the end of the previous section
we prove, in Theorem \ref{diastasis}, that this potential is indeed Calabi's diastasis function.

\begin{defin}\label{admissible}(\cite{A-P2}, Definition 8.1)
Let $F = G/K$, $G \subseteq GL(N, \bC)$, be a flag manifold. A principal minor $\Delta_k$, $k=1, \dots, N-1$, (i.e. the function $GL(N, \bC) \rightarrow \bC$ associating to $A \in GL(N, \bC)$ the determinant $\Delta_k(A)$ of the submatrix of $A$ given by the first $k$ rows and columns of $A$) is said to be $F$-{\it admissible} if for every $A \in K^{\bC}$ and every $v = (v_1, \dots, v_N) \in \bC^N$, $v_{k+1} = \cdots = v_N=0$ implies $(vA)_{k+1} = \cdots = (vA)_N = 0$.

\end{defin}

\begin{ex}\label{exsflags}

For the flag manifolds of the classical groups (see, for example, \cite{arv}) 

\begin{enumerate}

\item[] $G/K = SU(d)/S(U(d_1) \times \cdots \times U(d_s))$ ($d = d_1 + \cdots + d_s$, $s \geq 1$): 

\item[] $G/K = Sp(d)/U(d_1) \times \cdots \times U(d_s) \times Sp(l)$

\item[] $G/K = SO(2d+1)/U(d_1) \times \cdots \times U(d_s) \times SO(2l+1)$

\item[] $G/K = SO(2d)/U(d_1) \times \cdots \times U(d_s) \times SO(2l)$

($d = d_1 + \cdots + d_s + l$, $s, l \geq 0, l \neq 1$)

\end{enumerate}

\noindent it is easy to see that a minor $\Delta_k$ is admissible if and only if $k = d_1 + \cdots + d_j$, for some $j=1, \dots, s-1$ (resp. $j=1, \dots, s$) in the case $G=SU(d)$ (resp. in all the other cases).

\end{ex}

We have the following:

\begin{theor}\label{teorpotential}(\cite{A-P2}, Proposition 8.2)
Let $F = G/K$, $G= SU(d), Sp(d), SO(2d)$, $SO(2d+1)$, be a flag manifold represented by a painted Dynkin diagram endowed with the canonical equipment $\Pi_{can}$ given in the Appendix and let $\{ \alpha_{k_1}, \dots, \alpha_{k_m} \}$, $k_1 < \cdots < k_m$, be the set of black nodes, with associated fundamental weights $\bar \alpha_{k_i}$, $i=1, \dots, m$. Let $F$ be endowed with the $G$-invariant complex structure determined by $Q = R_M^+$. Then, in the holomorphic coordinates $z = (z_{\alpha})_{\alpha \in -Q}$ defined in (\ref{complexcoordinates}) on the open dense subset $F_{reg}$, a K\"ahler potential for the invariant K\"ahler form $\omega_{\xi}$, where $\xi = \sum_{j=1}^m c_j \bar \alpha_{k_j}$, $c_j > 0$,  is 

\begin{equation}\label{potential}
z = (z_{\alpha})_{\alpha \in -Q}  \mapsto \sum_{j=1}^m c_j \ln \Delta_{k_1 + \cdots + k_j}({}^T \overline{\exp(Z(z))} \exp(Z(z)) )
\end{equation}
 where 

\begin{equation}\label{Z(z)}
Z(z) = \sum_{\alpha \in -Q} z_{\alpha} E_{\alpha} 
\end{equation}

 and $\Delta_{k_1+\cdots+k_j}$ is the $j$-th $F$-admissible minor.

\end{theor}
Notice that, by Remark \ref{remarkalternativeQ}, the assumptions on the equipment and on $Q$ are not restrictive. In order to prove the main result of this section, Theorem \ref{diastasis} below, we need the following:

\begin{lem}\label{Lemmaexp}
Let $F = G/K$, $G= SU(d), Sp(d), SO(2d), SO(2d+1)$, be a flag manifold endowed with the complex structure associated to the maximal closed nonsymmetric subset $Q = R_M^+$ with respect to the canonical equipment $\Pi_{can}$. Let $Z(z) \in {\mathfrak{g}}^{-Q}$ given by (\ref{Z(z)}). Then, for every $i, j = 1, \dots, d$ we have:

\begin{enumerate}

\item[(a)] If the entry $\exp(Z(z))_{ij}$, $i \neq j$, is non-identically zero then the same is true for $Z(z)_{ij}$;

\item[(b)] $\exp(Z(z))_{ii} = 1$.

\end{enumerate}

\end{lem}

\begin{proof} Let us first deal with the case $G = SU(d)$. If $\exp(Z(z))_{ij}$, $i \neq j$, is non-identically zero, then there exists $k > 0$ such that $(Z(z)^k)_{ij} \neq 0$. So part $(a)$ will be proved if we show that if $(Z(z)^k)_{ij}$ is non-identically zero then the same is true for $Z(z)_{ij}$. Let us see this by induction on $k$. For $k=1$ this is obvious. For $k >1$, from $(Z(z)^k)_{ij} = \sum_{l=1}^d (Z(z)^{k-1})_{il} Z(z)_{lj} \neq 0$ it follows that there exists $l = 1, \dots, d$ such that $(Z(z)^{k-1})_{il} \neq 0$ and $Z(z)_{lj} \neq 0$. By the inductive assumption, we have $Z(z)_{il} \neq 0$ and $Z(z)_{lj} \neq 0$. But  for $Z(z) \in {\mathfrak{g}}^{-Q}$ we have $Z(z)_{ij} \neq 0$ $\Leftrightarrow e_i - e_j \in -Q$ (see the Appendix). So we have $e_i - e_l \in -Q$ and $e_l - e_j \in -Q$, and by Definition \ref{defQ} (iii) we have $e_i - e_j = (e_i - e_l) + (e_l - e_j) \in -Q$, which in turn implies $Z(z)_{ij} \neq 0$, as required.

 To show $(b)$, we prove that $(Z^k)_{ii} = 0$ for every $k >0$. For $k=1$, this follows from $Z(z) = \sum_{\alpha \in -Q} z_{\alpha} E_{\alpha}$ and from the definition of $E_{\alpha}$ given in the Appendix. Let $k > 1$: if in the sum $(Z(z)^k)_{ii} = \sum_{l=1}^d (Z(z)^{k-1})_{il} Z(z)_{li}$ we have $(Z(z)^{k-1})_{il} \neq 0$, $Z(z)_{li} \neq 0$ for some $l \neq i$,  it follows by $(a)$ that $Z(z)_{il} \neq 0$, so that $e_i-e_l, e_l-e_i \in -Q$, against Definition \ref{defQ} (ii). So $(Z(z)^k)_{ii} = (Z(z)^{k-1})_{ii} Z(z)_{ii} = 0$ by the case $k=1$. (Notice that the above arguments hold true for general maximal closed nonsymmetric subset $Q$).

 Let now $G = Sp(d), SO(2d), SO(2d+1)$. By the explicit matrix representation of the root vectors $E_{\alpha}$ given in the Appendix we see that, under the assumption $Q = R_M^+$, in all these cases the matrices $E_{\alpha}$, $\alpha \in -Q$ are block matrices of the kind $\left( \begin{array}{cc}
* & 0 \\
* & *  
\end{array} \right)$, so that we have $\exp(Z(z)) = \left( \begin{array}{cc}
\exp(\tilde Z(z)) & 0 \\
* & *  
\end{array} \right)$, where $\tilde Z(z) \in M_d(\bC)$ is the matrix having $z_{\alpha}$ in the place $ij$ if $e_i - e_j = \alpha \in -Q$ and $0$ elsewhere. Then, by the assumption $i, j = 1, \dots, d$, to conclude the proof of $(a)$, $(b)$ it is enough to show respectively that $(a')$ If the entry $\exp(\tilde Z(z))_{ij}$, $i \neq j$, is non-identically zero then the same is true for $\tilde Z(z)_{ij}$; $(b')$ $\exp(\tilde Z(z))_{ii} = 1$. But, since $\tilde Z(z)$ is constructed by using only the roots in $-Q$ of the kind $e_i - e_j$, these two claims are true by the same arguments used for the case $G = SU(d)$. 
\end{proof}

Now we are ready to prove the main result of this section.

\begin{theor}\label{diastasis}
Let $F$ be a flag manifold of classical type. 
Then the potential given by formula (\ref{potential}) is Calabi's diastasis function.
Moreover, $F_{reg}\subset F$ is the domain of maximal extension of Calabi's diastasis function.  
\end{theor}
\begin{proof}
Without loss of generality we can assume  $F=G/K$ is irreducible  and hence
$G= SU(d), Sp(d), SO(2d), SO(2d+1)$.
Recall that Calabi's diastasis function is the (unique determined) potential
around a point $p$ such that, in any given system of complex coordinates $z = (z_1, \dots, z_n)$ centered at $p$,
its power series development in $z$ and $\bar z$ does not contain the monomials $z^J$ or $\bar z^J$, for any nonzero multiindex $J = (j_1, \dots, j_n) \in (\bZ_{\geq 0})^n$ (see \cite{Cal}, \cite{loidiast} and \cite{diastexp}). 

 Let us denote by the brevity of notation $A = A(z, \bar z) = {}^T \overline{\exp(Z(z))} \exp(Z(z))$, $\Psi_l(z, \bar z) = \Delta_l(A)$ and $\psi_l = \ln \Psi_l$. We are going to prove that, when $l$ is chosen so that $\Delta_l$ is an $F$-admissible minor, then the power series development of $\psi_l$ around $z=0$ does not contain the monomials $z^J$, i.e. $\frac{\partial^{|J|} \psi_l}{\partial z^J}(0) = 0$ (the case of the monomials $\bar z^J$ is similar and we leave it to the reader).

Since $Z(0)$ is the null matrix, notice that we have $\exp(Z(0)) = I$ and $\Psi_l(0)= 1$. From this and from the fact (easily seen by induction) that 
$$\frac{\partial^{|J|} \psi_l}{\partial z^J} = \frac{1}{\Psi_l} \frac{\partial^{|J|} \Psi_l}{\partial z^J} + \textrm{(terms containing derivatives of $\Psi_l$ order} < |J|) $$
one sees that if the power series development of $\Psi_l$ at $z=0$ does not contain the terms $z^J$ then the same is true for the power series development of $\psi_l$.
So, assume by contradiction that $\frac{\partial^{|J|} \Psi_l}{\partial z^J}(0) \neq 0$ for some $J \neq 0$. Then, by the very definition of determinant $\Psi_l = \Delta_l(A) = \sum_{\sigma \in S_l} \epsilon(\sigma) A_{1 \sigma(1)} \cdots A_{l \sigma(l)}$, it is clear that there exist $\sigma \in S_l$ and $K_1, \dots, K_l \in (\bZ_{\geq 0})^n$ ($K_i = 0$ is allowed), $K_1+ \cdots + K_l = J$, such that $\frac{\partial^{|K_i|} A_{i \sigma(i)}}{\partial z^{K_i}}(0) \neq 0$ for every $i = 1, \dots, l$.

But, since $A_{ij} = \sum_s \overline{\exp(Z(z))}_{si} \exp(Z(z))_{sj}$ and for every $K \in (\bZ_{\geq 0})^n$ one has

$$\frac{\partial^{|K|} A_{ij}}{\partial z^{K}}(0) = \sum_s \overline{\exp(Z(0))}_{si} \frac{\partial^{|K|} \exp(Z(z))_{sj}}{\partial z^{K}}(0) =\frac{\partial^{|K|} \exp(Z(z))_{ij}}{\partial z^{K}}(0) ,$$

\noindent we conclude that there exist $\sigma \in S_l$ and $K_1, \dots, K_l \in (\bZ_{\geq 0})^n$, $K_1+ \cdots + K_l \neq 0$, such that 

\begin{equation}\label{expnonzero}
\frac{\partial^{|K_i|} \exp(Z(z))_{i \sigma(i)}}{\partial z^{K_i}}(0) \neq 0
\end{equation}
for every $i = 1, \dots, l$.

Now, by Lemma \ref{Lemmaexp} $(b)$, if $\sigma(i) = i$, then $\exp(Z(z))_{i \sigma(i)} = 1$, so that $K_i=0$ in this case. Since $K_1 + \cdots + K_l \neq 0$, it must exist $i_1 = 1, \dots, l$ such that $\sigma(i_1) \neq i_1$. Take the cyclic permutation $i_1, i_2 = \sigma(i_1), \dots, i_s = \sigma(i_{s-1}), i_1 = \sigma(i_s)$ starting at $i_1$. By (\ref{expnonzero})  we have 

\begin{equation}\label{expnonzero2}
\exp(Z(z))_{i_1 i_2}, \exp(Z(z))_{i_2 i_3}, \dots, \exp(Z(z))_{i_s i_1} \neq 0 .
\end{equation}

Since $\{ i_1, \dots, i_s \} \subseteq \{1, \dots, l \}$ and $\Delta_l$ is an admissible minor, as we have noticed in Examples (\ref{exsflags}) we have $l \leq d$ for all the classical groups \linebreak $G = SU(d), Sp(d), SO(2d), SO(2d+1)$ and then, by Lemma \ref{Lemmaexp} $(a)$, conditions (\ref{expnonzero}) imply
$$(Z(z))_{i_1 i_2}, (Z(z))_{i_2 i_3}, \dots, (Z(z))_{i_s i_1} \neq 0 .$$

 By definition of $Z(z)$ and by the description of the root vectors given in the Appendix, it means that $e_{i_1} -e_{ i_2}, e_{i_2} - e_{i_3}, \dots, e_{i_s} -e_{i_1} \in -Q$. By Definition \ref{defQ} (iii) we have $e_{i_2} -e_{ i_3}+ \cdots + e_{i_s} -e_{ i_1} = e_{i_2} - e_{i_1} \in -Q$ which, together with $e_{i_1} -e_{ i_2} \in -Q$, contradicts Definition \ref{defQ} (ii). 

Finally, in order to prove the last assertion in the statement of the theorem we will show that $D(z, \bar z) = \sum_{j=1}^m c_j \ln \Delta_{k_1 + \cdots + k_j}({}^T \overline{\exp(Z(z))} \exp(Z(z)) )$ on $\bC^n$ goes to infinity for $z = a \xi = (a_1 \xi, \dots, a_n \xi)$, when $\xi \in \bC$, $|\xi| \rightarrow \infty$, for every  $a \neq 0$. 

First, we claim that, for each admissible minor $\Delta_l$, the function \linebreak $\Delta_{l}({}^T \overline{\exp(Z(z))} \exp(Z(z)) )$ is a polynomial in $z, \bar z$. Indeed, under the assumption $Q = R_M^+$ with respect to the canonical equipment, the matrix $Z(z)$ is clearly nilpotent lower triangular in the case $G = SU(d)$, while in the cases \linebreak $G = Sp(d), SO(2d)$ (resp. $G = SO(2d+1)$)  is a block matrix of the kind \linebreak $\left( \begin{array}{cc}
\tilde Z(z) & 0 \\
* & - {}^T \tilde Z(z)  
\end{array} \right)$ (resp. of the kind $\left( \begin{array}{ccc}
\tilde Z(z) & 0 & 0 \\
* & - {}^T \tilde Z(z) & * \\
* & 0 & 0  
\end{array} \right)$), where $\tilde Z(z)$ is nilpotent lower triangular, which is easily seen to imply again that $Z(z)$ is nilpotent. So, in all the cases, the entries of $\exp(Z(z))$ are polynomials in $z$, which clearly implies the claim. 
Set $P_j(z, \bar z) = \Delta_{k_1 + \cdots + k_j}({}^T \overline{\exp(Z(z))} \exp(Z(z)) )$. The diastasis $D = \sum_{j=1}^m c_j \ln P_j = \ln (P_1^{c_1} \cdots P_m^{c_m})$, with $c_j > 0$, does not tend to infinity in the direction $z = a \xi$ if and only if $P_j(a \xi, \bar a \bar \xi)$, $j = 1, \dots, m$ stays bounded, which for a polynomial is possibile only when it is constant. This implies that the potential itself is constant along the same direction, i.e. $\phi(\xi) := D(a \xi, \bar a \bar \xi)$ is constant. But this contradicts the fact that $D$ is a potential for a \K\ metric, since we have

$$0 = \frac{\partial^2 \phi}{\partial \xi \partial \bar \xi} = \sum_{i,j=1}^n \frac{\partial^2 D}{\partial z_i \partial \bar z_j}(a \xi, \bar a \bar \xi) a_i \bar a_j$$
which is not possible if $a \neq 0$.
\end{proof}

\begin{remark}\rm
\rm We believe that our results generalize to the homogeneous manifolds of the exceptional groups $E_i$ ($i=6,7,8$), $F_4$, $G_2$ (see, for example, Table 4 in \cite{BFR} for a complete list in terms of painted Dynkin diagrams). In particular,  in \cite{DHL}, by using the theory of Jordan triple systems, it was proved that the symmetric spaces of both classical and exceptional groups admit a system of complex coordinates defined on an open dense subset biholomorphic to $\bC^n$.
\end{remark}

\section{Proof of Theorem \ref{mainteor}}\label{secmain}
In order to prove Theorem \ref{mainteor} we need the following two lemmata which show how to 
extend $L^2$-holomorphic functions to holomorphic sections of holomorphic   line bundles.


\begin{lem}\label{extensionsection}
Let $M$ be a compact Hodge manifold of complex dimension $n$ with \K\ form $\omega$
and let $(L, h)$ be a Hermitian line bundle such that $\Ric (h)=\omega$.
Assume that there exists a divisor $Y\subset M$ such  that the restriction of  $L$ to ${M\setminus Y}$ 
is the  trivial holomorphic line bundle 
and let $\sigma :M\setminus Y\rightarrow L$ be a trivializing holomorphic section.
Let $f$ be a holomorphic function on $M\setminus Y$ such that
$$\int_{M\setminus Y} |f (x)|^2h(\sigma (x), \sigma (x))
\frac{\omega ^{n}}{n!}<\infty .$$
Then $f$ extends to a  (unique) global holomorphic section, namely there exists $s\in H^0(L)$ such that 
$s(x)=f(x)\sigma (x)$ for all $x\in M\setminus Y$.
\end{lem}
\proof
Choose   trivializing holomorphic sections  $\tau_{\alpha} :U_\alpha\rightarrow L$ defined  on open and connected open subsets $U_\alpha\subset M$
such that $Y\subset\cup_\alpha U_\alpha$.
Let $g_\alpha :U_\alpha\setminus Y\rightarrow \bC$  be the holomorphic function such that $\tau_\alpha (x) =g_\alpha(x) \sigma (x)$ for $x\in U_\alpha\setminus Y$.
It follows by the assumption that, for each $\alpha$,
$$\int_{U_\alpha\setminus Y} \frac{|f (x)|^2}{|g_\alpha(x)|^2} h(\tau_\alpha (x), \tau_\alpha (x))
\frac{\omega ^{n}}{n!}=\int_{U_\alpha\setminus Y} |f (x)|^2h(\sigma (x), \sigma (x))
\frac{\omega ^{n}}{n!}<\infty .$$
Fix such an $\alpha$.
Thus, the function  $\frac{f}{g_\alpha }:U_\alpha\setminus Y$ 
is  a  $L^2$-bounded holomorphic function  on $U_\alpha\setminus Y$.
Since the set $Y\cap U_\alpha$ is an analytic subset of $U_\alpha$ it follows by \cite[Thm. 5.17, p. 101]{ohsawa}
that $\frac{f}{g_\alpha}$ extends to a holomorphic function to all of $U_\alpha$.
The desired global holomorphic  section $s\in H^0(L)$ 
is then given by:
$$s(x)= \left\{\begin{array}{l}
f(x)\sigma (x)\ \  \ \  {\rm if \ }\ \ \ x\in M\setminus Y, \\
\frac{f(x)}{g_\alpha (x)}\tau_{\alpha} (x) \ \    {\rm if \ }\ \  x\in Y\cap U_\alpha.
\end{array}\right.$$
\endproof

\begin{lem}\label{extensionform}
Let  $N$ be a (not necessarily compact) complex manifold   of dimension $m$ and  let $Z$ be an analytic subset of $N$
such that $V=N\setminus Z$ can be equipped  with complex coordinates $w_1, \dots , w_m$.
Let $f$ be a holomorphic function on $V$ such that 
$$\int_{V} |f(x)|^2d\nu<\infty, $$
where 
$d\nu :=\frac{i^m}{2^m}dw_1\wedge d\ov w_1\wedge\cdots\wedge dw_m\wedge d\ov w_m$.
Then there exists a $(m, 0)$-form $\eta$  on $N$
such that 
 $\eta_{|V}=f dw_1\wedge\cdots\wedge dw_m$.
\end{lem}
\proof
The proof can be obtained by an  argument similar  to that of  the previous lemma using the line bundle of $\Lambda^{m,0}K$
whose sections are holomorphic $m$-forms  on $N$.
\endproof

Another important ingredient in the proof of Theorem \ref{mainteor}
is the concept of regular quantization.
Let $(L, h)$ be a positive Hermitian line bundle over a compact \K\ manifold $(M, g)$ of complex dimension $n$, such that $\Ric (h)=\omega_g$ as in the introduction.
Let
$m\geq 1$ be an integer and
consider the Kempf distortion function associated to $mg$, i.e.
\begin{equation}\label{Tm}
T_{mg} (x) =\sum_{j=0}^{d_m}h_m(s_j(x), s_j(x))
\end{equation}
where  $h_m$ is an hermitian metric   on
$L^m$ such that $\Ric(h_m)=m\omega_g$
and   $s_0, \dots , s_{d_m}$, $d_m+1=\dim H^0(L^m)$ is  an orthonormal basis of $H^0(L^m)$
(the space of holomorphic sections of $L^m$)
with respect to the $L^2$-scalar product
\begin{equation}\label{scalprod}
\langle s, t \rangle_m =\int_Mh_m(s(x), t(x))\frac{\omega_g^n(x)}{n!}, \ s, t\in H^0(L^m).
\end{equation}
(In the  quantum geometric context
$m^{-1}$ plays the role of Planck's constant, see e.g.
\cite{arlquant}).
As suggested by the notation this function depends only on the
\K\ metric $g$ and on $m$ and not on the orthonormal basis chosen.

The function $T_{mg}(x)$ admits the following asymptotic expansion (called Tian-Yau-Zelditch expansion or, for short, TYZ expansion)
\begin{equation}\label{asymptoticZ}
T_{mg}(x) \sim \sum_{j=0}^\infty  a_j(x)m^{n-j},
\end{equation}
where  $a_j(x)$, $j=0,1, \ldots$, are smooth coefficients with $a_0(x)=1$.
More precisely,
for any nonnegative integers $r,k$ the following estimate holds:
\begin{equation}\label{rest}
||T_{mg}(x)-
\sum_{j=0}^{k}a_j(x)m^{n-j}||_{C^r}\leq C_{k, r}m^{n-k-1},
\end{equation}
where $C_{k, r}$ are constants depending on $k, r$ and on the
K\"{a}hler form $\omega_g$ and $ || \cdot ||_{C^r}$ denotes  the $C^r$
norm in local coordinates. 
In \cite{lu} Z. Lu,  by means of  Tian's peak section method,
 proved  that each of the coefficients $a_j(x)$ in
(\ref{asymptoticZ}) is a polynomial
of the curvature and its
covariant derivatives at $x$ of the metric $g$ which can be found
 by finitely many algebraic operations.
 Furthermore,  he explicitely computes
$a_j$ with $j\leq 3$. The reader is referred to \cite{ALZ}
and references therein for details on Kempf distortion function and its 
link with the TYZ  expansion.

Prescribing the values of  the  coefficients
of the  TYZ expansion gives rise to interesting elliptic PDEs
as shown by Z. Lu and G. Tian  \cite{logterm}. The main result obtained there,  is that if the log-term
of the Bergman kernel  of the unit disk bundle over $M$ defined in the introduction  vanishes then  $a_k=0$, for  $k>n$
($n$ being the complex dimension of $M$). Moreover Z. Lu has conjectured (private communication) that the converse is true, namely if 
$a_k=0$, for  $k>n$ then the log-term vanishes.

A \K\ manifold  $(M, g)$ admits a {\em regular quantization} if  
there exists  a positive Hermitian line bundle $(L, h)$ as above  such that 
the Kempf distortion $T_{mg} (x)$ is  a constant $T_{mg}$ (depending on $m$)
for all non-negative integer   $m\geq 1$. 
One also refers to $(L, h)$ as a regular quantization of $(M, g)$.

Regular quantizations play a prominent role in the study of Berezin quantization of \K\ manifolds  (see  \cite{arlquant} , \cite{arlcomm}, \cite{hombal} and  \cite{berloimossa} and references therein).
From our point of view we are interested in the following result.
\begin{lem}\label{homreg}
Any  geometric quantization $(L, h)$ of a  compact,  simply-connected and    homogeneous \K\ manifold $(M, g)$
is  regular.
\end{lem}
\begin{proof}
See \cite[Theorem 5.1]{arlquant}.
\end{proof}
We are now in the position to prove Theorem \ref{mainteor}.

\begin{proof}[Proof of Theorem \ref{mainteor}]
A  compact homogeneous Hodge manifold  $(M, g)$ is a flag manifold $(F, \omega)$,
where $\omega=\omega_{\xi}$  is an invariant \K\ form given by \eqref{omegaxi} 
and $\omega/\pi$  is integral. 
It follows by Theorem  \ref{teorpotential}
that  there exists a dense and open subset $U=F_{reg}\subset M$ biholomorphic to $\bC^n$ 
(with $n$ the complex dimension of $M$) and a global  \K\ potential $\Phi :U\rightarrow \R$
for the metric $g$, i.e. $\omega_g=\frac{i}{2}\partial\bar\partial\Phi$.
Moreover, since we are assuming that  $(M, g)$ is of classical type, it  follows by  Theorem \ref{diastasis}
that $\Phi$ is  Calabi's diastasis function  
for the the metric $g$ and $\Phi$ blows up on the points of  $Y=M\setminus U$.
On the other hand, $(M, g)$ can be \K\ embedded into some complex projective space (see  Section 1.2. in
\cite{DHL}), namely there exists a positive integer $N$ and a  holomorphic embedding $f:M\rightarrow \bC P^N$
such that $f^*g_{FS}=g$, where $g_{FS}$ denotes the Fubini--Study metric of $\bC P^N$.
 Hence, by a theorem of Calabi \cite{Cal},  $Y$  turns out to be the  divisor on $M$
given by the pull-back via the embedding $f$ of the hyperplane divisor of $\bC P^N$.
As $U$ is contractible the restriction of $L^*$ to $U$ is holomorphically trivial, $L^*_{|U}\cong U\times \bC$.
Let $\sigma :U\rightarrow L$ be a trivializing holomorphic section of $L$ and let $\tau:U\rightarrow L^*$
be the dual section trivializing $L^*$, namely $\tau (x)(\sigma (x))=1$.
Then 
$h^*(\tau (x), \tau (x))=\left(h(\sigma (x), \sigma (x))\right)^{-1}$ 
and
it follows by \eqref{diskbundle} that  the restriction of the disc bundle $D$ to $U$ is biholomorphic to the Hartogs domain
$H\subset\bC^{n+1}$ given by:
\begin{equation}
H=\{\left( z ,\, \lmb \right) \in \bC^{n+1} \ |\ \left| \lmb \right| ^2 < h \left( z \right) ,\ z \in \bC^n    \},
\end{equation}
where $h(z):=h(\sigma (x), \sigma (x))$
and we identify a point $x\in U$
with its coordinates $z=(z_1, \dots , z_n)\in\bC^n$.
Moreover, the restriction of $\partial D$ to $U$ is identified with
$$\partial H=\{\left( z ,\, \lmb \right) \in \bC^{n+1} \ |\ \left| \lmb \right| ^2 = h \left( z \right) \}.$$

By Lemma \ref{extensionform} (with $N=D$, $m=n+1$, $Z=\pi^{-1}(Y)$, $\pi :D\rightarrow M$ the bundle projection and hence $V=H$)
the Bergman space $\mathcal B _{D}$ can be identified with 
the (usual) Bergman space 
\[
\mathcal L^2(H)=\{f\in\Hol (H) \ | \|f\|^2= \frac{i}{2} \int_H|f(z, \lambda)|^2\ d\lmb\wedge d\ov\lmb\wedge d\mu <\infty\},
\]
where $d\mu :=\frac{i^n}{2^n}dz_1\wedge d\ov z_1\wedge\cdots\wedge dz_n\wedge d\ov z_n$.
Let  $K(z, \lambda)$ be  the reproducing kernel  of $\mathcal L^2(H)$.
In view of \eqref{logtermzero}, we need to show that there exists  a non vanishing smooth function 
 $a\in C^{\infty}(\ov H)$  such that  
 \begin{equation}\label{dadim}
 K(z, \lambda)=a(z, \lambda)\rho (z, \lmb)^{-n-2},\ \ \rho (z, \lambda)=1-\frac{|\lambda|^2}{h(z)}.
 \end{equation}

The Hilbert space  ${\mathcal L^2}\left( H\right)$ admits the Fourier decomposition into irreducible factors with respect to  the natural  $S^1$-action $s(z, \lambda)\mapsto (z, s\lambda), s\in S^1$
and hence we can write 
\[
{\mathcal L}^2\left( H\right) =\oplus_{m=0}^{+\infty}\,{\mathcal L^2_m}\left( H\right),  
\]
where 
$${\mathcal L^2_m}\left( H\right) =\{f\in {\mathcal L}^2\left( H\right)  \ | \ f\left(z, s\lambda \right) =s^m\,f\left(z, \lambda\right) ,\  s\in S^1 \}.$$

Now, the assumption that $\frac{\omega_g}{\pi} \in c_1(K^*)$ implies that the Einstein constant of $g$ equals $2$, i.e.
$$\rho_{\omega_g}=-i\de\deb\log (\det g)=2\omega_g .$$
By \eqref{prodherm},  $\frac{\omega_g^n}{n!}=\det (g)d\mu$ and the $\de\deb$-lemma there exists a
holomorphic function $\varphi$ on $U$ such that 
$\frac{\omega_g^n}{n!}=e^{\varphi +\ov\varphi}\,h\,d\mu.$

\noindent If $\{ s_j^m \}$ is an orthonormal basis of $(H^0(L^m), \langle\cdot  ,  \cdot\rangle_m)$ (given by  \eqref{scalprod}), let us define functions $f_j\left(z, \, \lmb \right) \in \mathcal L^2_m \left( H\right)$ by $f_j\left(z, \, \lmb \right) = \lmb^m f_j\left(z \right)$ where
$$f_j(z)\sigma_m (z)=\sqrt\frac{\left( m+1 \right) }{\pi} \ e^{\varphi}{s^{m}_j}(z),$$
\noindent and $\sigma_m$
is the trivializing section of $L^m$ given by $\sigma_m=\sigma^{\otimes m}$.  Notice that 

\[
\frac{i}{2} \int_H f_j(z, \lambda) \ov f_k(z, \lambda)\ d\lmb\wedge d\ov\lmb\wedge d\mu=\frac{i}{2} 
\int_{\bC^n}\int_{\left| \lmb \right| ^2 < h\left( z \right)}  \left| \lmb \right| ^{2\, m}   f_j\left( z\right) \ov f_k\left( z\right)    \ d\lmb\wedge d\ov\lmb\wedge d\mu= 
\]  
\[
=\frac{\pi}{m+1}\int_{\bC^n} h \left( z \right) ^{m+1} f_j\left( z\right) \ov f_k\left( z\right) d\mu =\frac{\pi}{m+1}\int_{\bC^n} e^{-\varphi -\ov\varphi} h \left( z \right) ^{m} f_j\left( z\right) \ov f_k\left( z\right) \frac{\omega_g^n}{n!}=
\]
\begin{equation}\label{pim1}
=\int_{\bC^n} h_m(s^m_j(z), s^m_k(z))  \frac{\omega_g^n}{n!}= \delta_{jk}.
\end{equation}

Combining  this computation with Lemma \ref{extensionsection} we see that $\{ f_j(z, \lmb) \}_{j=0, \dots, N_m}$ is an orthonormal basis of $(\mathcal L^2_m \left( H\right), \| \cdot \|)$. Thus

\[
K_m \left(z, \, \lmb \right) 
= \sum_{j=0}^{N_m} \left| f_j \left(z, \, \lmb \right)\right| ^2 
=   \left( \frac{\left|\lmb\right|^2}{h \left( z \right) } \right) ^m\sum_{j=0}^{N_m} h \left( z \right) ^{m} \left| f_j \left(z\right)\right| ^2
\]
\[=
 \left( \frac{\left|\lmb\right|^2}{h \left( z \right) } \right) ^m \,  \frac{ \left( m+1 \right)  e^{\varphi + \ov \varphi}}{\pi} \, T_{mg}\left( z \right),
\]
where $T_{mg}$ is Kempf distortion function given by \eqref{Tm}.
It follows  by Lemma \ref{homreg}  and by the very definition of Kempf distortion function that 
\begin{equation}\label{Tmh}
T_{mg}(z)=T_{mg}=\frac{h^{0}(L^m)}{V(M)},
\end{equation}
where $V(M)=\int_M\frac{\omega^n}{n!}$.
By Riemann--Roch theorem  $h^0(L^m)$ is a monic  polynomial in  $m$  of degree $n$ so it can be written as  linear combination
of the binomial coefficients  $C_{k}^{m+k}=\frac{(m+k)!}{m!\, k!}$, namely
\[
h^0(L^m)=\sum_{k=0}^nd_k\, C_{k}^{m+k},\quad  d_n=n!.
\]
Observe that
\[
m\,C^{m+k}_k=\left( k + 1 \right) C^{\left( m-1 \right) + \left( k+1 \right)}_{k+1} 
\]
and
\[
\sum_{m=0}^{\infty}C_{k}^{m+k}\, x^m=\frac{1}{\left( 1-x\right) ^{k+1}}, \ 0<x<1.
\]
Hence
\[
K\left(z, \, \lmb \right)= \sum _ {m=0} ^ \infty K_m \left(z, \, \lmb \right)
=\frac{ e^{\varphi + \ov \varphi}}{  \pi \, V(M)} \sum_{k=0}^n   \, d_k   \sum _ {m=0} ^ \infty { \left( m+1 \right) \, C_{k}^{m+k} }  \left( \frac{\left|\lmb\right|^2}{h \left( z \right) } \right) ^m
\]
\[
=\frac{e^{\varphi + \ov \varphi}}{ \pi \, V(M)} \sum_{k=0}^n  \, d_k   \sum _ {m=0} ^ \infty  \left( \left( k + 1 \right) C^{\left( m-1 \right) + \left( k+1 \right)}_{k+1} + C_{k}^{m+k} \right) \left( \frac{\left|\lmb\right|^2}{h \left( z \right) } \right) ^m
\]
\[
=\frac{e^{\varphi + \ov \varphi}}{  \pi \, V(M)} \sum_{k=0}^n  \, d_k  \left( \left(  k + 1 \right)  \rho\left(  z, \, \lmb\right)^{- k - 2}  +   \rho\left(  z, \, \lmb\right)^{- k - 1}        \right) =a(z, \lambda)\rho\left(  z, \, \lmb\right)^{-n-2},
\]
where
\[
a(z, \lambda)=\frac{e^{\varphi + \ov \varphi}}{\pi \, V(M)} \sum_{k=0}^n  \, d_k  \left( \left(  k + 1 \right)  \rho\left(  z, \, \lmb\right)^{n - k}   +  \rho\left(  z, \, \lmb\right)^{n-k+1}        \right).  
\]
This proves \eqref{dadim} and concludes the proof of the theorem (notice that when $\rho\left(  z, \, \lmb\right) \rightarrow 0$  then $a(z, \lambda) \rightarrow \frac{e^{\varphi + \ov \varphi}}{\pi \, V(M)} d_n (n+1) \neq 0$, so $a$ does not vanish on $\partial D$.
\end{proof}

\appendix
\section{Classical Lie Algebras}\label{projectiveQ}
We now describe root systems, root vectors and Dynkin diagram  for the classical groups $G = SU(d), Sp(d), SO(2d), SO(2d+1)$
following \cite[Section III.8]{helgason}.
The diagrams are  endowed with the equipment $\Pi_{can}$ we have used throughout the paper and which we have referred to as the canonical equipment. 

\vskip 0.3cm

\begin{ex}\rm
If $G= SU(d)$,  $\mathfrak{g}^{\bC} = sl(d, \bC)$ is the set of matrices with null trace, and let $\mathfrak{h}^{\bC}$ be given by the diagonal matrices in $sl(d, \bC)$; for any $H = \diag(h_1, \dots, h_d)$ let $e_i(H) = h_i$: then the root system is $R = \{ e_i - e_j \ | \ i \neq j \}$ and $E_{\alpha}$, $\alpha = e_i - e_j$, is the matrix $E_{ij}$ having 1 in the ij place and 0 anywhere else. The Killing form $B$ satisfies $B(X,Y)= 2d \tr(XY)$ for all $X, Y \in sl(d, \bC)$.
The canonical basis is 
$$\Pi_{can} = \{ \alpha_1 = e_1 - e_2, \dots, \alpha_{d-1} = e_{d-1} - e_d \}.$$ 
The Dynkin diagram, with this equipment, is
\begin{align*}
 \underset{\substack{\alpha_1}}{\circ} - \underset{\substack{\alpha_2}}{\circ} - \dotsb - \underset{\substack{\alpha_{d-2}}}{\circ} - \underset{\substack{\alpha_{d-1}}}{\circ} 
 \\
\end{align*}
\end{ex}
\begin{ex}\rm
For $G= Sp(d)$, $\mathfrak{g}^{\bC} = sp(d, \bC)$ is the set of $2d \times 2d$ block matrices of the kind $ \left( \begin{array}{cc}
Z_1 & Z_2 \\
Z_3 & -{}^T Z_1 \\
\end{array} \right)$, where $Z_2, Z_3$ are symmetric. Let the Cartan subalgebra $\mathfrak{h}^{\bC}$ be given by diagonal matrices $H = \diag(h_1, \dots, h_d, -h_1, \dots, -h_d)$ in $sp(d, \bC)$, and if for any such $H$ we have $e_i(H) = h_i$, $i=1, \dots, d$, then the root system is $R = \{ \pm e_i \pm e_j \}$ (the case $i=j$ is allowed when the signs are equal). The root vector $E_{\alpha}$ is given by

$\left( \begin{array}{cc}
E_{ij} & 0 \\
0 & -E_{ji} \\
\end{array} \right)$ if $\alpha = e_i - e_j$, $\left( \begin{array}{cc}
0 & E_{ij} + E_{ji} \\
0 & 0 \\
\end{array} \right)$ if $\alpha = e_i + e_j$ and \linebreak $\left( \begin{array}{cc}
0 & 0 \\
E_{ij} + E_{ji} & 0 \\
\end{array} \right)$ if $\alpha = -e_i - e_j$
and the Killing form $B$ is given by $B(X,Y)= 2(d+1) \tr(XY)$ for all $X, Y \in sp(d, \bC)$.

The canonical basis is given by 
$$\Pi_{can} = \{ \alpha_1 = e_1 - e_2, \dots, \alpha_{d-1} = e_{d-1} - e_d, \alpha_d = 2 e_d \}.$$ 
The Dynkin diagram, with this equipment, is 
\begin{align*}
\underset{\substack{\alpha_1}}{\circ} - \underset{\substack{\alpha_2}}{\circ} - \dotsb - \underset{\substack{\alpha_{d-1}}}{\circ} \Leftarrow \underset{\substack{\alpha_d}}{\circ} \\
\end{align*}
\end{ex}

\begin{ex}\rm
Let $G= SO(2d)$. Here and throughout the paper we identify the complexification $SO(2d, \bC)$ with the subgroup of $GL(2d, \bC)$ leaving invariant the quadratic form $z_1 z_{d+1} + \cdots + z_d z_{2d}$. Then $\mathfrak{g}^{\bC} = so(2d, \bC)$ is the set of $2d \times 2d$ block matrices of the kind $\left( \begin{array}{cc}
Z_1 & Z_2 \\
Z_3 & -{}^T Z_1 \\
\end{array} \right)$, where $Z_2, Z_3$ are skew-symmetric. Let the Cartan subalgebra $\mathfrak{h}^{\bC}$ be given by diagonal matrices \linebreak $H = \diag(h_1, \dots, h_d, -h_1, \dots, -h_d)$ in $so(2d, \bC)$, and if for any such $H$ we have $e_i(H) = h_i$, $i=1, \dots, d$, then the root system is $R = \{ \pm e_i \pm e_j \ (i \neq j) \}$. The root vector $E_{\alpha}$ is given by

$\left( \begin{array}{cc}
E_{ij} & 0 \\
0 & -E_{ji} \\
\end{array} \right)$ if $\alpha = e_i - e_j$, $\left( \begin{array}{cc}
0 & E_{ij} - E_{ji} \\
0 & 0 \\
\end{array} \right)$ if $\alpha = e_i + e_j$ ($i<j$) and $\left( \begin{array}{cc}
0 & 0 \\
E_{ij} - E_{ji} & 0 \\
\end{array} \right)$ if $\alpha = -e_i - e_j$ ($i<j$)
and the Killing form $B$ is given by $B(X, Y) =2(d-1)\tr(XY)$ for all $X ,Y \in so(2d, \bC)$.

The canonical basis is given by 
$$\Pi_{can} = \{ \alpha_1 = e_1 - e_2, \dots, \alpha_{d-1} = e_{d-1} - e_d, \alpha_d = e_{d-1} + e_d \}.$$ 
The Dynkin diagram, with this equipment, is
\begin{align*}
 \underset{\substack{\alpha_1}}{\circ} - \underset{\substack{\alpha_2}}{\circ} - \dotsb - \underset{\substack{\alpha_{d-2}}}{\overset{\overset{\textstyle\circ_{\alpha_d}}{\textstyle\vert}}{\circ}} \,-\, \underset{\substack{\alpha_{d-1}}}{\circ} && \\
\end{align*}
\end{ex}

\begin{ex}\rm
Let $G= SO(2d+1)$. Here and throughout the paper we identify the complexification $SO(2d+1, \bC)$ with the subgroup of $GL(2d+1, \bC)$ leaving invariant the quadratic form $2(z_1 z_{d+1} + \cdots + z_d z_{2d}) + z_{2d+1}$. Then $\mathfrak{g}^{\bC} = so(2d+1, \bC)$ is the set of $(2d+1) \times (2d+1)$ block matrices of the kind $\left( \begin{array}{ccc}
Z_1 & Z_2 & u \\
Z_3 & -{}^T Z_1 & v \\
-{}^T v & -{}^T u & 0
\end{array} \right)$, where $Z_2, Z_3$ are skew-symmetric and $u, v \in \bC^d$. Let the Cartan subalgebra $\mathfrak{h}^{\bC}$ be given by diagonal matrices \linebreak $H = \diag(h_1, \dots, h_d, -h_1, \dots, -h_d, 0)$ in $so(2d+1, \bC)$, and if for any such $H$ we have $e_i(H) = h_i$, $i=1, \dots, d$, then the root system is $R = \{ \pm e_i \pm e_j \ (i \neq j), \pm e_i \}$. The root vector $E_{\alpha}$ is given by

$\left( \begin{array}{ccc}
E_{ij} & 0 & 0 \\
0 & -E_{ji} & 0\\
0 & 0 & 0 
\end{array} \right)$ if $\alpha = e_i - e_j$, $\left( \begin{array}{ccc}
0 & E_{ij} - E_{ji} & 0 \\
0 & 0 & 0 \\
0 & 0 & 0
\end{array} \right)$ if $\alpha = e_i + e_j$ ($i<j$), $\left( \begin{array}{ccc}
0 & 0 & 0 \\
E_{ij} - E_{ji} & 0 & 0 \\
0 & 0 & 0
\end{array} \right)$ if $\alpha = -e_i - e_j$ ($i<j$), $\left( \begin{array}{ccc}
0 & 0 & E_i \\
0 & 0 & 0 \\
0 & -{}^T E_i & 0
\end{array} \right)$ if $\alpha = e_i $ and $\left( \begin{array}{ccc}
0 & 0 & 0 \\
0 & 0 & E_i \\
-{}^T E_i & 0  & 0
\end{array} \right)$ if $\alpha = - e_i $ (being $E_i$ the $i$-th vector of the canonical basis of $\bC^d$)
and the Killing form $B$ is given by $B(X, Y) =(2d-1)\tr(XY)$ for all $X ,Y \in so(2d+1, \bC)$.

The canonical basis is given by 
$$\Pi_{can} = \{ \alpha_1 = e_1 - e_2, \dots, \alpha_{d-1} = e_{d-1} - e_d,  \alpha_d = e_d \}.$$ The Dynkin diagram, with this equipment, is 

\begin{align*}
&&& \underset{\substack{\alpha_1}}{\circ} - \underset{\substack{\alpha_2}}{\circ} - \dotsb - \underset{\substack{\alpha_{d-1}}}{\circ} \Rightarrow \underset{\substack{\alpha_d}}{\circ} && \\
\end{align*}
\end{ex}

\begin{remark}\label{remaximaltorus}
\rm In each of the above examples, the given Cartan subalgebra is the complexification of the Lie algebra of a maximal torus $T$ in the compact group $G$ (more precisely, for $G = SU(d), Sp(d), SO(2d), SO(2d+1)$ we have respectively $T = \diag(e^{i \theta_1}, \dots, e^{i \theta_d}), \diag(e^{i \theta_1}, \dots, e^{i \theta_d}, e^{-i \theta_1}, \dots, e^{-i \theta_d})$, 

$\diag(e^{i \theta_1}, \dots, e^{i \theta_d}, e^{-i \theta_1}, \dots, e^{-i \theta_d})$, $\diag(e^{i \theta_1}, \dots, e^{i \theta_d}, e^{-i \theta_1}, \dots, e^{-i \theta_d}, 1)$). In general, given a maximal torus $T$ in a compact connected Lie group $G$, its Lie algebra $\mathfrak{t}$ is a maximal abelian subalgebra of $\mathfrak{g}$ and its complexification $\mathfrak{t}^{\bC}$ is a Cartan subalgebra of the complex Lie algebra $\mathfrak{g}^{\bC}$.
\end{remark}

\end{document}